\documentclass[12pt,a4paper,reqno]{amsart}

\usepackage{amssymb,amsmath,graphicx}
\usepackage{fancyvrb}
\usepackage{booktabs}
\usepackage{pgf}
\usepackage{tikz}
\usetikzlibrary{fit}
\newcommand\addvmargin[1]{%
  \node[fit=(current bounding box),inner ysep=#1,inner xsep=0]{};
}
\usepackage{longtable}
\usepackage{enumerate}
\usepackage{comment}
\usetikzlibrary{arrows,automata}
\usetikzlibrary{decorations.pathreplacing}

\newtheorem{theorem}{Theorem}[section]
\newtheorem{lemma}[theorem]{Lemma}
\newtheorem{corollary}[theorem]{Corollary}
\newtheorem{proposition}[theorem]{Proposition}
\newtheorem{conjecture}[theorem]{Conjecture}

\theoremstyle{definition}
\newtheorem{definition}[theorem]{Definition}
\newtheorem{example}[theorem]{Example}

\newtheorem{remark}[theorem]{Remark}

\newtheorem{notation}[theorem]{Notation}

\newcommand{\Jonq}{Jonqui\`eres}
\newcommand{\BB}{\mathcal{B}}
\newcommand{\PP}{\mathbb{P}}
\newcommand{\K}{\mathbb{K}}

\newcommand{\infnear}{\succ}
\newcommand{\prox}{\dasharrow}

\newcommand{\satel}{\odot}

\DeclareMathOperator{\Aut}{Aut}

\DeclareMathOperator{\Bir}{Bir}
\DeclareMathOperator{\PGL}{PGL}

\DeclareMathOperator{\q}{ql}
\DeclareMathOperator{\lgth}{lgth}

\DeclareMathOperator{\id}{id}
\DeclareMathOperator{\mult}{mult}

\numberwithin{table}{section}
\numberwithin{figure}{section}

\title[On the quadratic length of plane Cremona maps of degree 4]{On the quadratic length of \\ plane Cremona maps of degree 4}
\author{Alberto Calabri$^{*}$}
\address{$^{*}$Dipartimento di Matematica e Informatica, Università di Ferrara, Via Machiavelli 30, 44121 Ferrara, Italy}
\email{clblrt@unife.it}

\author{Nguyen Thi Ngoc Giao$^{\dag \ddag}$}
\address{$^{\dag}$Department of Mathematics, Faculty of Science and Technology, Tokyo University of Science, 2641 Yamazaki, Noda, Chiba, 278-8510, Japan}
\address{$^{\ddag}$Faculty of Electrical Engineering, The University of Danang - University of Science and Technology, 54 Nguyen Luong Bang, Da Nang, Vietnam}
\email{ngocgiao185@gmail.com}

\begin{document}

\begin{abstract}
Every non-linear plane Cremona map can be decomposed into quadratic maps, and the minimum number of quadratic maps required is called its quadratic length. It is known that plane Cremona maps of degree 3 have quadratic length either 2 or 3.

In this paper, we study the quadratic length of plane Cremona maps $\varphi$ of degree 4.
Recall that either $\varphi$ is \emph{de Jonquières}, i.e.\ it has a base point of multiplicity 3, or it is not de Jonquières.
In the latter case, $\varphi$ has quadratic length 2, whereas in the former case we prove that $\varphi$ has quadratic length 3, 4, or 5, and we classify those maps that reach the upper bound.
\end{abstract}

\maketitle


\section{Introduction}

We work over an algebraically closed field $\K$ of characteristic zero.

We denote by $\PP^2$ the projective plane over $\K$ and by $\Bir(\PP^2)$ the \emph{plane Cremona group}, that is, the group of birational maps $\PP^2\dasharrow\PP^2$.
According to the celebrated Noether-Castelnuovo Theorem, any plane Cremona map $\varphi$ (of degree larger than one) can be written as the composition of quadratic maps.

Let us call the \emph{quadratic length} of $\varphi$ the minimum number of quadratic maps needed to decompose $\varphi$.
Recall that, in \cite{B-F}, Blanc and Furter defined and studied the (Jonquières) \emph{length} of $\varphi$ as the minimum number of de \Jonq\ maps needed to decompose $\varphi$.
Note that the Jonquières length of a plane Cremona map $\varphi$ only depends on the so-called homaloidal type of $\varphi$, namely on the multiplicities of its base points, while the quadratic length also depends on the position of its (infinitely near) base points, as shown by Theorem \ref{thm3.1} below.

Let us say that two plane Cremona maps are \emph{equivalent} if one can be obtained from the other by changing projective coordinates in the source and/or the target plane.

A classification of equivalence classes of plane Cremona maps of degree 3 has been described only very recently; see \cite{C-D} and \cite{CN}.
In particular, in \cite{CN} we gave a fine classification of them and, as a byproduct, we proved the following theorem:

\begin{theorem}\label{thm3.1}
Let $\varphi$ be a plane Cremona map of degree 3.
Then, the quadratic length of $\varphi$ is either $2$ or $3$. In particular, $\varphi$ has quadratic length $3$ if and only if $\varphi$ is equivalent to the following map
\begin{equation}\label{phi1}
\PP^2 \dashrightarrow \PP^2,
\qquad
[x:y:z] \mapsto [xz^2 +y^3 : yz^2 : z^3].
\end{equation}
\end{theorem}

Recall that plane Cremona maps of degree 4 either are de Jonquières or their homaloidal type is $(4;2^3,1^3)$.
In the former case, they have Jonquières length 1, while in the latter case, they have Jonquières length 2.
It is easy to see (cf.\ Lemma \ref{ql=2} below) that all plane Cremona maps of degree 4 and Jonquières length 2 have quadratic length 2.

The main result of this paper is the following theorem:

\begin{theorem}\label{thm1.2}
Let $\varphi$ be a plane Cremona map of degree 4 and de Jonquières.
Then, the quadratic length of $\varphi$ is $3$, $4$ or $5$. In particular, $\varphi$ has quadratic length $5$ if and only if $\varphi$ is equivalent to one of the following two maps
\begin{equation*}
\PP^2 \dashrightarrow \PP^2,
\qquad
[x:y:z] \mapsto [xz^3+y^4+ky^2z^2 : yz^3 : z^4].
\end{equation*}
where $k$ is either $0$ or $1$.
\end{theorem}

The main ingredient of the proof is the analysis of the so-called weighted proximity graph associated to a plane Cremona map (see Section 4 for its definition).
Indeed, Theorem \ref{thm1.2} turns out to be equivalent to the following one.

\begin{theorem}\label{thm1.3}
Let $\varphi$ be a plane Cremona map of degree 4.
Then, the quadratic length of $\varphi$ is at most 5
and it is 5 if and only if the weighted proximity graph of $\varphi$ is:
\[
\scalebox{0.6}{%
\begin{tikzpicture}[->,>=stealth',shorten >=2pt,auto,node distance=2.5cm,ultra thick,baseline=0]
  \tikzstyle{every state}=[fill=white,draw=black,text=black]
  \node[state,draw=red,text=red] (A)    {\huge 3};
  \node[state] (B) [right of=A] {\huge 1};
  \node[state] (C) [right of=B] {\huge 1};
  \node[state] (D) [right of=C] {\huge 1};
  \node[state] (E) [right of=D] {\huge 1};
  \node[state] (F) [right of=E] {\huge 1};
  \node[state] (G) [right of=F] {\huge 1};
\path (B) edge (A);
\path (C) edge (B);
\path (D) edge (C);
\path (E) edge (D);
\path (F) edge (E);
\path (G) edge (F);
\path (C) edge [bend right] (A);
\path (D) edge [bend right] (A);
\addvmargin{1mm}
\end{tikzpicture}%
}
\]
\end{theorem}

\medskip
A natural generalization of Theorem \ref{thm1.3} is the following:

\begin{conjecture}
Let $\varphi$ be a plane Cremona map of degree $d$.
Then, the quadratic length of $\varphi$ is at most $2d-3$
and it is $2d-3$ if and only if the weighted proximity graph of $\varphi$ has the following properties:
\begin{itemize}
    \item $2d-1$ vertices $p_1, p_2, \ldots, p_{2d-1}$;
    \item weight $d-1$ at $p_1$ and weight 1 at $p_i$, $i=2,\ldots,2d-1$;
    \item for each $i=2,\ldots,2d-1$, there is an arrow $p_i \to p_{i-1}$;
    \item for each $j=3,\ldots,d$, there is an arrow $p_j\to p_1$.
\end{itemize}
\end{conjecture}

It would be nice to have an algorithm that computes the quadratic length of a plane Cremona map.
The asymptotic properties of the quadratic length are also interesting. We will investigate on these problems in future.

Let us briefly explain the contents of this paper.

In Section \ref{sec:prel} we recall the definition and properties of infinitely near points and we review the definition of standard coordinates that we introduced in \cite{CN} to describe them.

In Section \ref{sec:plane_Cremona_maps} we recall the basic definitions regarding plane Cremona maps, their homaloidal type and their lengths.

In Section \ref{sec:weighted} we define the enriched weighted proximity graphs for a plane Cremona map of degree 4 and Jonquières length 1.

In Section \ref{sec:height} we study quadratic length of plane Cremona maps of degree 4 and we prove Theorem \ref{thm1.2}.

\section{Notation, preliminaries and infinitely near points}
\label{sec:prel}

Let us recall the definition of the bubble space of $\PP^2$, useful to define infinitely near points and related properties.

\begin{definition}[{cf.\ \cite[\S7.3.2]{Dolgachev}, \cite[Chapter 5, \S35]{M}}]
Let us denote by $\BB(\PP^2)$ the so-called \emph{bubble space} of $\PP^2$, which is defined as follows.
Consider all surfaces $X$ \emph{above} $\PP^2$, i.e.\ all surfaces $X$ such that there exists a birational morphism $X\to\PP^2$.
If $X_1,X_2$ are two surfaces above $\PP^2$, say $\pi_1\colon X_1\to\PP^2$ and $\pi_2\colon X_2\to\PP^2$ are birational morphisms,
one identifies $p_1\in X_1$ with $p_2\in X_2$ if the birational map $(\pi_2)^{-1}\circ \pi_1\colon X_1\dasharrow X_2$ is a local isomorphism at $p_1$ that sends $p_1$ to $p_2$.
The bubble space $\BB(\PP^2)$ is the union of all points of all surfaces above $\PP^2$ modulo the equivalence relation generated by these identifications.

For any birational morphism $X\to\PP^2$, there is an injective map $X\to\BB(\PP^2)$, therefore we will identify points of $X$ with their images in $\BB(\PP^2)$.

One says that $p_1\in\BB(\PP^2)$ is \emph{infinitely near} $p_2\in\BB(\PP^2)$, say $p_1\in X_1$ and $p_2\in X_2$, with birational morphisms $\pi_1\colon X_1\to\PP^2$ and $\pi_2\colon X_2\to\PP^2$, if the birational map $(\pi_2)^{-1}\circ\pi_1\colon X_1\dasharrow X_2$ is defined at $p_1$, sends $p_1$ to $p_2$, but is not a local isomorphism at $p_1$. In such a case we write that $p_1 \infnear p_2$.

One moreover says that $p_1$ is \emph{in the first neighbourhood} of $p_2$, or that $p_1$ is \emph{infinitely near $p_2$ of the first order}, if $(\pi_2)^{-1}\circ\pi_1$ corresponds locally to the blow-up of $p_2$.
In such a case we write that $p_1 \infnear_1 p_2$.

If $p_1 \infnear p_2$ then one defines the \emph{infinitesimal order} of $p_1$ with respect to $p_2$ by induction, namely if $p_1 \infnear_1 p_3$ and $p_3 \infnear_k p_2$ for some $k$, then $p_1$ is \emph{infinitely near $p_2$ of order $k+1$}. 

If $p_1\infnear p_2$ and $p_1\in X_1$, then there is a unique irreducible curve $E_2\subset X_1$ which corresponds to the exceptional curve of the blowing-up of $p_2\in X_2$. One says that $p_1$ is \emph{proximate} to $p_2$ if $p_1\in E_2$. In such a case we write that $p_1\prox p_2$.

If $p_1\prox p_2$ and $p_1\infnear_k p_2$ with $k >1$, then one says that $p_1$ is \emph{satellite} to $p_2$ and we write $p_1 \satel p_2$. Otherwise, if $p_1$ is not satellite to any other point, then one says that $p_1$ is \emph{free}.

One says that a point $p\in\PP^2\subset\BB(\PP^2)$ is a \emph{proper point of $\PP^2$}.

If $p_1\infnear p_2 \in \PP^2$ where $p_1 \in X_1$ and $\pi_1\colon X_1 \rightarrow \PP^2$ is a birational morphism, we say that a plane curve $C$ passes through $p_1$ if  the strict transform of $C$ on $X_1$ via $\pi_1$ passes through $p_1$.
\end{definition}

\begin{definition}[{cf.\ \cite[\S3]{CN}}]
We recall the \emph{standard coordinates} of infinitely near points that we introduced in \cite{CN}.

Let $p_1=[a:b:c] \in \PP^2$. Let us consider three cases:
\begin{itemize}
\item[$(i)$] if $c \neq 0$, then $p_1 =\bigg[\dfrac{a}{c}:\dfrac{b}{c}:1\bigg] = [\overline{a}:\overline{b}:1]$;
\item[$(ii)$] if $c = 0$ and $b \neq 0$, then   $p_1 =\bigg[\dfrac{a}{b}:1:0\bigg] = [\overline{a}:1:0]$;
\item[$(iii)$] if $c=b= 0$, then $p_1 = [1:0:0]$.
\end{itemize}

In case $(i)$, we work on the affine chart $U_2 \simeq \K^2_{\overline{x},\overline{y}}$, so that $p_1$ corresponds to the point $\overline{p}_1 = (\overline{a},\overline{b})$, and we define the isomorphism $\alpha_1 \colon \K^2_{\overline{x},\overline{y}} \rightarrow \K^2_{x_0,y_0}$ by
$\alpha_1(\overline{x},\overline{y}) = (\overline{x}-\overline{a},\overline{y}-\overline{b}).$

In case $(ii)$, we work on the affine chart $U_1 \simeq \K^2_{\overline{x},\overline{z}}$, so that $p_1$ corresponds to the point $\overline{p}_1 = (\overline{a},0)$, and we define the isomorphism $\alpha_1 \colon \K^2_{\overline{x},\overline{z}} \rightarrow \K^2_{x_0,y_0}$ by
$\alpha_1(\overline{x},\overline{z}) = (\overline{x}-\overline{a},\overline{z}).$

In case $(iii)$, we work on the affine chart $U_0 \simeq \K^2_{\overline{y},\overline{z}}$, so that $p_1$ corresponds to the point $\bar p_1=(0,0)$, and we define the isomorphism $\alpha_1 \colon \K^2_{\overline{y},\overline{z}} \rightarrow \K^2_{x_0,y_0}$ by
$\alpha_1(\overline{y},\overline{z}) = (\overline{y},\overline{z}).$

In all three cases, we defined $\alpha_1$ in such a way that $\alpha_1(\overline{p}_1) = (0,0)\in \K^2_{x_0,y_0}$.

We blow-up $\K^2_{x_0,y_0}$ at $(0,0)$ and we consider the first chart $\K^2_{x_1,y_1}$ where the blowing-up map is given in coordinates by $x_0=x_1, y_0=x_1y_1$.

In this chart, the exceptional curve $E_1$ has local equation $x_1=0$, hence a point $p_2 \infnear_1 p_1$ corresponds either to the point $(0,t_2) \in E_1$ with $t_2 \in \K$ or to the point which is the origin of the second chart.
In the former case, let us say that $p_2$ has standard coordinates $p_2=(p_1,t_2)$, while in the latter case let us say that $p_2$ has standard coordinates $p_2=(p_1,\infty)$.
Setting $\PP^1=\K\cup\{\infty\}$, in both cases we may write $p_2=(p_1,t_2)$ with $t_2\in\PP^1$.

We go on by blowing-up at $p_2=(p_1,t_2)$, with $t_2\in\PP^1=\K\cup\{\infty\}$.
Either $t_2\in\K$ or $t_2=\infty$. In the former case, with notation as above, let $\alpha_2 \colon \K^2_{x_1,y_1} \rightarrow \K^2_{\bar x_1,\bar y_1}$ be the isomorphism defined by
$\alpha_2(x_1,y_1) = (x_1, y_1-t_2).$

In the latter case, $p_2$ corresponds to the origin of the second chart of the blowing-up of $\K^2_{x_0,y_0}$ at $(0,0)$ that we write  $\K^2_{x'_1,y'_1}$, where the blowing-up map is given by $x_0=x'_1 y'_1, y_0=y'_1$. 
Let $\alpha_2 \colon \K^2_{x'_1,y'_1} \rightarrow \K^2_{\bar x_1,\bar y_1}$ be the isomorphism
$\alpha_2(x'_1,y'_1) = (y'_1, x'_1).$

In this way, in both cases, in $\K^2_{\bar x_1,\bar y_1}$ the exceptional curve $E_1$ has local equation $\bar x_1=0$ and the point $p_2$ corresponds to the origin $(0,0)$.

We blow-up $\K^2_{\bar x_1,\bar y_1}$ at $(0,0)$ and we consider the first chart $\K^2_{x_2,y_2}$ where the blowing-up map is given in coordinates by $\bar x_1=x_2, \bar y_1=x_2y_2$.
In this chart, the exceptional curve $E_2$ has local equation $x_2=0$, hence a point $p_3 \infnear_1 p_2$ corresponds either to the point $(0,t_3) \in E_2$ with $t_3 \in \K$ or to the point which is the origin of the second chart.

Let us say that $p_3$ has standard coordinates $p_3=(p_1,t_2,t_3)$, where either $t_3\in\K$ in the former case or $t_3=\infty$ in the latter case.

Note that the strict transform of $E_1$ can be seen only in the second chart and it meets $E_2$ at the origin of the second chart.
In other words, the point with standard coordinates $(p_1,t_2,\infty)$ is satellite to $p_1$.

More generally, let us proceed by induction on the infinitesimal order.
Suppose that we have blown-up the point $p_{r-1}$ with standard coordinates $p_{r-1}=(p_1,t_2,\ldots,t_{r-1})$, with $t_i\in\PP^1=\K\cup\{\infty\}$, $i=2,\ldots,r-1$.
Following the procedure described above, we may assume that $p_{r-1}$ is the origin of a chart $\K^2_{\bar x_{r-1},\bar y_{r-1}}$ in such a way that the exceptional curve $E_{r-1}$ has local equation $\bar x_{r-1}=0$.

In the first chart of the blowing up of $\K^2_{\bar x_{r-1},\bar y_{r-1}}$ at $(0,0)$, given in coordinates by $\bar x_{r-1}=x_r, \bar y_{r-1}=x_ry_r$, the exceptional curve $E_r$ has local equation $x_r=0$, hence a point $p_r \infnear_1 p_{r-1}$ corresponds either to the point $(0, t_r) \in E_r$ with $t_r\in\K$ or to the point which is the origin of the second chart, given in coordinates by $\bar x_{r-1}=x_ry_r, \bar y_{r-1}=y_r$.

Let us say that $p_r$ has standard coordinates $p_r=(p_1,t_2,\ldots,t_r)$, where $t_r\in\K$ in the former case and $t_r=\infty$ in the latter case.

In such a way, we are able to define the standard coordinates for each point of $\BB(\PP^2)$ as given by a point in $\PP^2$ and a sequence of points of $\PP^1$.
\end{definition}

The following lemma has been proved in \cite{CN}.

\begin{lemma}
Suppose that $p_1, p_2, p_3, p_4, p_5 \in \BB(\PP^2)$ are such that no three of them are collinear and either
\begin{itemize}
\item $p_1,p_2,p_3,p_4,p_5 \in \PP^2$; or
\item $p_1,p_2,p_3,p_4 \in \PP^2$ and $p_5 \infnear_1 p_1$; or
\item $p_1,p_2,p_3 \in \PP^2$, $p_5 \infnear_1 p_4 \infnear_1 p_1$ and $p_5$ is free; or
\item $p_1,p_2,p_3 \in \PP^2$, $p_4 \infnear_1 p_1$, and $p_5 \infnear_1 p_2$; or
\item $p_1,p_2 \in \PP^2$, $p_5 \infnear_1 p_3 \infnear_1 p_1$, $p_4 \infnear_1 p_2$ and $p_5$ is free; or
\item $p_1,p_2 \in \PP^2$, $p_5 \infnear_1 p_4 \infnear_1 p_3 \infnear_1 p_1$ and $p_4, p_5$ are free; or
\item $p_5 \infnear_1 p_4 \infnear_1 p_3 \infnear_1 p_2 \infnear_1 p_1 \in \PP^2$ and $p_3, p_4, p_5$ are free.
\end{itemize}
Then, there exists a unique irreducible plane conic passing through $p_1,\ldots, p_5$.
\end{lemma}

\section{Plane Cremona maps and their lengths}\label{sec:plane_Cremona_maps}

Recall that a plane Cremona map $\varphi\colon \PP^2 \dasharrow \PP^2$ can be written as
\begin{equation*}
\varphi( [x : y : z] ) = [f_0(x, y, z) : f_1(x, y, z) : f_2(x, y, z)]
\end{equation*}
where $f_i\in\K[x,y,z]$, $i=0,1,2$, are homogeneous polynomials of the same degree, say $d$, that is called the \emph{degree} of $\varphi$ if $f_0, f_1, f_2$ have no common factor of positive degree.
Plane Cremona maps of degree 1 are automorphisms of $\PP^2$, i.e.\ elements of $\Aut(\PP^2)\cong\PGL(3, \K)$.
For convenience, we refer to the plane Cremona maps of degree 2 as \emph{quadratic maps}.

\begin{definition}
Let us say that  two plane Cremona maps $\varphi,\varphi'\colon\PP^2\dasharrow\PP^2$  are \emph{equivalent} if there exist two automorphisms $\alpha,\alpha'\in\Aut(\PP^2)$ such that $\varphi'=\alpha'\circ\varphi\circ\alpha$.
\end{definition}

\begin{remark}
It is very well-known that quadratic Cremona maps are equivalent to one and only one of the following three maps:
\begin{align}
\sigma([x:y:z]) &= [yz : xz : xy], \label{sigma} \\
\rho([x:y:z]) &=[xy : z^2 : yz], \label{rho} \\
\tau([x:y:z]) &=[x^2 : xy : y^2-xz]. \label{tau}
\end{align}
Note that $\sigma, \rho, \tau$ are \emph{involutions}, i.e.\ $\sigma^{2}=\rho^{2}=\tau^{2}=\id_{\PP^2}$.
More generally, a plane Cremona map $\varphi$ is called an \emph{involution}, or also an \emph{involutory map}, if $\varphi^2=\id_{\PP^2}$ and $\varphi\ne\id_{\PP^2}$.
\end{remark}

The following lemma is very well-known but we have not found a proper reference.

\begin{lemma}\label{involutory}
Suppose that there exists a quadratic Cremona map based at three points $q_1, q_2, q_3\in\BB(\PP^2)$, namely either one of the following three cases occurs:
\begin{enumerate}
  \item $q_1, q_2, q_3 \in \PP^2$;
  \item $q_1, q_2 \in \PP^2$ and $q_3\infnear_1 q_1$;
  \item $q_3 \infnear_1 q_2 \infnear_1 q_1\in\PP^2$ and $q_3$ is free.
\end{enumerate}
Then there exists an involutory quadratic map $\psi$ based at $q_1, q_2, q_3$.
\end{lemma}

\begin{proof}
Recall that the three base points of $\sigma$ are the coordinate points $p_1=[1:0:0]$, $p_2=[0:1:0]$, $p_3=[0:0:1]$, while the three base points of $\rho$ are $p_1$, $p_2$ and the point with standard coordinates $(p_1,\infty)$, i.e.\ the infinitely near point to $p_1$ in the direction of the line $y=0$, and finally the standard coordinates of the base points of $\tau$ are respectively $p_3$, $(p_3,\infty)$, $(p_3,\infty,1)$, i.e.\ the infinitely near point of order 2 to $p_3$ in the direction of the plane conic $y^2-xz=0$ (the point $(p_3,\infty)$ is the infinitely near point to $p_3$ in the direction of the line $x=0$).

In case $(1)$, let $\alpha\in\Aut(\PP^2)$ be an automorphism which maps $q_1, q_2, q_3$ to the three coordinate points $p_1$, $p_2$, $p_3$; the map $\psi=\alpha^{-1}\circ \sigma \circ \alpha$ satisfies the required property.

In case $(2)$, let $\beta\in\Aut(\PP^2)$ be an automorphism which maps $q_1$ to $p_1$, $q_2$ to $p_2$, $q_3$ to $(p_1,\infty)$; the map $\psi=\beta^{-1}\circ \rho \circ \beta$ satisfies the required property.

In case $(3)$, let $\gamma\in\Aut(\PP^2)$ be an automorphism which maps $q_1$ to $p_3$, $q_2$ to $(p_3,\infty)$, $q_3$ to $(p_3,\infty,1)$; the map $\psi=\gamma^{-1}\circ \tau \circ \gamma$ satisfies the required property.
\end{proof}

Recall that the linear system defining a plane Cremona map $\varphi$ of degree $d$ has base points $p_1,\ldots,p_r\in\BB(\PP^2)$, of respective multiplicity $m_1,\ldots,m_r$,
that satisfy the following equations:
\begin{equation}\label{eq:Cremona}
d^2-1=\sum_{i=1}^r m_i^2,
\qquad\qquad
3(d-1)=\sum_{i=1}^r m_i.
\end{equation}
One says that $(d;m_1,\ldots,m_r)$ is the homaloidal type of $\varphi$.

A plane Cremona map is called \emph{de Jonqui\`eres} if it has degree $d\geq2$ and a base point of multiplicity $d-1$, i.e.\ its homaloidal type is $(d;d-1,1,\ldots,1)=(d;d-1,1^{2d-2})$.

According to the Noether-Castelnuovo Theorem, any plane Cremona map $\varphi\colon\PP^2\dasharrow\PP^2$ of degree larger than one can be written as the composition of quadratic maps:
\begin{equation}\label{eq:decomp2}
\varphi=\psi_n\circ\psi_{n-1}\circ \cdots \circ \psi_2\circ \psi_1.
\end{equation}

\begin{definition}
Let us call the \emph{quadratic length} of plane Cremona map $\varphi$ the minimum $n$ such that there exists a decomposition \eqref{eq:decomp2} where $\psi_i$ is a quadratic map, for each $i=1,\ldots,n$, and denote it by $\q(\varphi)$.
\end{definition}

Recall that Blanc and Furter in \cite{B-F} defined the (Jonquières) \emph{length} of a plane Cremona map $\varphi$ as the minimum $n$ such that there exists a decomposition \eqref{eq:decomp2} where $\psi_i$ is a de Jonqui\`eres map, for each $i=1,\ldots,n$, and denoted it by $\lgth(\varphi)$.
Clearly, one has that
$\lgth(\varphi) \leq \q(\varphi)$
and two equivalent plane Cremona maps have the same Jonquières length and the same quadratic length.

\begin{remark}\label{rem:inverse}
Taking the inverse maps, a decomposition \eqref{eq:decomp2} of a plane Cremona map $\varphi$ determines a decomposition of the inverse map $\varphi^{-1}$, and conversely.
In particular, a plane Cremona map $\varphi$ and its inverse $\varphi^{-1}$ have the same quadratic length: $\q(\varphi)=\q(\varphi^{-1})$.
\end{remark}

Note that Equations \eqref{eq:Cremona} imply that plane Cremona maps of degree 2 and 3 are de Jonqui\`eres, whereas those of degree 4 either are de Jonqui\`eres or their homaloidal type is $(4;2,2,2,1,1,1)=(4;2^3,1^3)$.
In the former case, they have Jonquières length 1, while in the latter case, they have Jonquières length 2.

\begin{lemma}\label{ql=2}
Plane Cremona maps of degree 4 and Jonquières length 2 have quadratic length 2 too.
\end{lemma}

\begin{proof}
Let $\varphi$ be a plane Cremona map of degree 4 and Jonquières length 2.
Let $p_1, p_2, p_3$ be the base points of multiplicity 2 of $\varphi$. Note that $p_1, p_2, p_3$ cannot be aligned by Bézout's theorem and, if $p_3 \infnear_1 p_2 \infnear_1 p_1$, then $p_3$ cannot be satellite to $p_1$.
Lemma \ref{involutory} implies that there exists an involutory quadratic map $\psi_1$ based at $p_1, p_2, p_3$. It follows that $\varphi\circ\psi_1$ is a quadratic map $\psi_2$, so that $\varphi=\psi_2\circ\psi_1$ and $\varphi$ has quadratic length 2.
\end{proof}

Therefore, in the following sections, we will deal only with plane Cremona maps of degree 4 that are de Jonquières, or equivalently that have Jonquières length 1.

\begin{remark}\label{rem:3.7}
Note that an automorphism $\alpha\colon\PP_1^2 \to \PP_2^2$, where the index is added to distinguish the source and the target plane, induces a bijection $\bar{\alpha}\colon\BB(\PP_1^2) \to \BB(\PP_2^2)$ which preserves the proximity relations, namely $p_j$ is proximate to $p_i$ in $\BB(\PP_1^2)$ if and only if $\bar{\alpha}(p_j)$ is proximate to $\bar{\alpha}(p_i)$ in $\BB(\PP_2^2)$.

The quadratic map $\sigma\colon \PP_1^2 \dasharrow \PP_2^2$, given in \eqref{sigma}, also induces a bijection $\bar{\sigma}\colon\BB(\PP_1^2) \to \BB(\PP_2^2)$.
Indeed, recall that $\sigma$ is the composition of the blowing up at the three coordinates points $p_1=[1:0:0]$, $p_2=[0:1:0]$, $p_3=[0:0:1]$ with the blowing down of the strict transform of the coordinate lines $\ell_1\colon x=0$, $\ell_2\colon y=0$, $\ell_3\colon z=0$.
In other words, setting $\pi_1\colon S_1\to\PP^2_1$ and $\pi_2\colon S_2\to\PP^2_2$ the birational morphisms which are the blowing up at $p_1$, $p_2$, $p_3$ respectively of $\PP^2_1$ and of $\PP^2_2$, there is an isomorphism $\tilde{\sigma}\colon S_1\to S_2$ such that $\sigma\circ\pi_1=\pi_2\circ\tilde{\sigma}$.
More explicitly, the bijection $\bar{\sigma}\colon\BB(\PP_1^2) \to \BB(\PP_2^2)$ is such that
\begin{itemize}
  \item if $p\in\PP^2_1\setminus (\ell_1 \cup \ell_2 \cup \ell_3)$, then $\bar{\sigma}(p)=\sigma(p)$;
  \item if $p\in\ell_i\setminus\{p_j,p_k\}$, $\{i,j,k\}=\{1,2,3\}$, then $\bar{\sigma}(p) \infnear_1 p_i$ in the direction of the line $\sigma(\ell)$, where $\ell$ is the line passing through $p_i$ and $p$;
  \item if $p\infnear_1 p_i$ in the direction of the line $\ell_j$, where $\{i,j\}\subset\{1,2,3\}$, then $\bar{\sigma}(p) \infnear_1 p_j$ in the direction of the line $\ell_i$;
  \item if $p\infnear_1 p_i$ (not lying on $\ell_j$, $\ell_k$, $\{i,j,k\}=\{1,2,3\}$), then $\bar{\sigma}(p)$ is a proper point of $\ell_i$;
  \item if $p=p_i$, $i=1,2,3$, then we set $\bar{\sigma}(p)=p$.
\end{itemize}
By the Noether-Castelnuovo Theorem, it follows that any plane Cremona map $\varphi\colon \PP_1^2 \dasharrow \PP_2^2$ induces a bijection $\bar{\varphi}\colon\BB(\PP_1^2) \to \BB(\PP_2^2)$, where we use standard coordinates to describe infinitely near points.
\end{remark}

Let us recall how a quadratic map acts on the homaloidal types of plane Cremona maps.

\begin{proposition}\label{PropMAC118b}
Let $p_1, p_2, p_3$ be the base points of an involutory quadratic map $\psi\colon \mathbb{P}^2 \dasharrow \mathbb{P}^2$.
Let $\varphi\colon \mathbb{P}^2 \dashrightarrow \mathbb{P}^2$ be a plane Cremona map of degree $d>1$ with base points $p_4,\ldots,p_r$ and possibly $p_1,p_2,p_3$. Denote by $m_i$ the multiplicity of $\varphi$ at $p_i$, $i=1,\ldots,r$
(that is $m_i=0$ if $p_i$ is not a base point of $\varphi$, $i=1,2,3$). Then, setting $\varepsilon = d-m_1-m_2-m_3$, the map $\varphi\circ \psi$ has
\begin{itemize}
  \item degree $d+\varepsilon$;
  \item for $i=1,2,3$, multiplicity $m_i+\varepsilon$  at $p_i$ (that is, $p_i$ is not a base point when $\varepsilon=-m_i$);
  \item base points of multiplicity $m_i$ for $i \in \{ 4,\ldots,r \}$;
  \item no other base point.
\end{itemize}
\end{proposition}

\begin{proof}
Cf.\ Proposition 4.2.5 in \cite{MAC}.
\end{proof}

\section{Enriched weighted proximity graphs}\label{sec:weighted}

\begin{definition}\label{def weighted_proximity_graph}
Let $\varphi$ be a plane Cremona map. Let us associate to $\varphi$ a \emph{weighted digraph} $G_\varphi$, called the \emph{weighted proximity graph of (the base points of) $\varphi$}, defined as follows:
\begin{itemize}
\item the vertices of $G_\varphi$ are the base points $p_1,\ldots,p_r \in \BB(\PP^2)$ of $\varphi$;
\item there is an arrow $p_i \to p_j$ if and only if $p_i$ is proximate to $p_j$;
\item each vertex $p_i$ is weighted with the multiplicity $m_i=\mult_{p_i}(\varphi)$ of $\varphi$ at $p_i$.
\end{itemize}
\end{definition}

\begin{example}
Let $\sigma$, $\rho$ and $\tau$ be the quadratic maps defined in \eqref{sigma}, \eqref{rho} and \eqref{tau}. Their respective proximity graphs $G_\sigma$, $G_\rho$ and $G_\tau$ are:
\begin{center}
$G_\sigma$ =
\scalebox{0.4}{%
\begin{tikzpicture}[->,>=stealth',shorten >=2pt,auto,node distance=2.5cm,ultra thick,baseline=-1.25ex]
  \tikzstyle{every state}=[fill=white,draw=black,text=black]
  \node[state,draw=red,text=red] (A)                    {\huge 1};
  \node[state,draw=red,text=red] (B) [right of=A] {\huge 1};
  \node[state,draw=red,text=red] (C) [right of=B] {\huge 1};
  \end{tikzpicture}%
} \qquad
$G_\rho$ =
\scalebox{0.4}{%
\begin{tikzpicture}[->,>=stealth',shorten >=2pt,auto,node distance=2.5cm,ultra thick,baseline=-1.25ex]
  \tikzstyle{every state}=[fill=white,draw=black,text=black]

  \node[state,draw=red,text=red] (A)                    {\huge 1};
  \node[state] (B) [right of=A] {\huge 1};
  \node[state,draw=red,text=red] (C) [right of=B] {\huge 1};
  \path (B) edge (A);
\end{tikzpicture}%
} \qquad
$G_\tau$ =
\scalebox{0.4}{%
\begin{tikzpicture}[->,>=stealth',shorten >=2pt,auto,node distance=2.5cm,ultra thick,baseline=-1.25ex]
  \tikzstyle{every state}=[fill=white,draw=black,text=black]

  \node[state,draw=red,text=red] (A)                    {\huge 1};
  \node[state] (B) [right of=A] {\huge 1};
  \node[state] (C) [right of=B] {\huge 1};

\path (C) edge (B);
\path (B) edge (A);
\end{tikzpicture}%
}
\end{center}
\end{example}

\begin{notation}
When we draw the weighted proximity graph of a plane Cremona map $\varphi$, for readers' convenience we write in \textcolor{red}{red} the vertices corresponding to \emph{proper base points} and in \textcolor{black}{black} the vertices corresponding to \emph{infinitely near points}.
\end{notation}

In \cite{CN} we proved that there are exactly 21 weighted proximity graphs of plane Cremona maps of degree 3.
Recall that in \cite{CN} we enriched the weighted proximity graphs of plane Cremona maps of degree 3 by adding the line passing through three base points:
we found exactly 31 enriched weighted proximity graphs of plane Cremona maps of degree 3.

\begin{example}
The weighted proximity graph of a plane Cremona map of degree 4 and Jonquières length 1 has one vertex of multiplicity 3 and six vertices of multiplicity 1.
\end{example}

\begin{remark}\label{rem:line}
Note that, if $\varphi$ has degree 4 and Jonquières length 1, then Bézout's theorem implies that:
\begin{itemize}
  \item there is no irreducible conic through the seven base points of $\varphi$;
  \item there is no line through five (or more) base points of $\varphi$;
  \item there is no line through the base point of multiplicity 3 and other two points of multiplicity 1 of $\varphi$.
\end{itemize}
\end{remark}

\begin{definition}
Let us add to the weighted proximity graph $G_\varphi$ of a plane Cremona map $\varphi$ of degree 4 and Jonquières length 1 the following data about lines and conics:
\begin{itemize}
    \item the list of lines passing through three or four base points of $\varphi$;
    \item the list of conics passing through six base points of $\varphi$.
\end{itemize}
Let us call this object the \emph{enriched weighted proximity graph} of $\varphi$.
\end{definition}

In \cite{thesis}, the second author proved that:
\begin{itemize}
    \item there are exactly 90 weighted proximity graphs of plane Cremona maps of degree 4 and Jonquières length 1, up to isomorphism;
    \item there are exactly 449 enriched weighted proximity graphs of plane Cremona maps of degree 4 and Jonquières length 1, up to isomorphism.
\end{itemize}

In this paper we are interested in computing lower and upper bounds for the quadratic length of plane Cremona maps of degree 4, so we do not put here the whole list of 449 enriched weighted proximity graphs.

\section{Quadratic length of plane Cremona maps of degree 4}
\label{sec:height}

Let us first recall some very well-known facts.

\begin{lemma}\label{cor5.14}\label{lem5.15}
A plane Cremona map $\varphi$ of degree $d\leq5$ and de Jonquières has $\q(\varphi) \geq d-1$.
A plane Cremona map $\varphi$ of degree $d\geq5$ has $\q(\varphi) \geq 3$.
\end{lemma}

\begin{proof}
See Corollary 5.14 and Lemma 5.15 in \cite{CN}.
\end{proof}

\begin{lemma}\label{lem:5.4}
Let $\varphi$ be a plane Cremona map of degree 4 and Jonquières length 1.
Suppose that $\varphi$ has at least one proper base point of multiplicity 1.
Then, $\q(\varphi) \leq 4$. 
\end{lemma}

\begin{proof}
Let $p_1$ be the base point of multiplicity 3 of $\varphi$ and let $p_2$ be a proper base point of multiplicity 1 of $\varphi$, that exists by the assumption.
Then, there are three possibilities:
\begin{itemize}
    \item there exists another proper base point $p_3$ (of multiplicity 1) of $\varphi$;
    \item there exists a base point $p_3 \infnear_1 p_1$;
    \item there exists a base point $p_3 \infnear_1 p_2$.
\end{itemize}
In all three cases, $p_1, p_2, p_3$ cannot be collinear, hence there exists an involutory quadratic map $\psi$ based at $p_1, p_2, p_3$ by Lemma \ref{involutory} and $\varphi\circ\psi$ is a plane Cremona map of degree 3 by Proposition \ref{PropMAC118b}.
By Theorem \ref{thm3.1}, one has that $\q(\varphi\circ\psi)\leq 3$. We conclude that $\q(\varphi)\leq 4$.
\end{proof}

\begin{lemma}\label{lem:5.5}
Let $\varphi$ be a plane Cremona map of degree 4 and Jonquières length 1.
Let $p_1\in\PP^2$ be the base point of multiplicity 3 of $\varphi$.
Suppose that there exist two other base points $p_2, p_3$ of $\varphi$ such that $p_3\infnear_1 p_2 \infnear_1 p_1$ and $p_3$ is not proximate to $p_1$.
Then, $\q(\varphi) \leq 4$. 
\end{lemma}

\begin{proof}
Note that $p_1, p_2, p_3$ cannot be collinear.
Therefore, there exists an involutory quadratic map $\psi$ based at $p_1, p_2, p_3$ and we may conclude as in the proof of the previous lemma.
\end{proof}

\begin{lemma}
Let $\varphi$ be a plane Cremona map of degree 4 and Jonquières length 1.
Let $p_1\in\PP^2$ be the base point of multiplicity 3 of $\varphi$.
Suppose that the hypotheses of Lemmas \ref{lem:5.4} and \ref{lem:5.5} are not fulfilled, namely all other base points of $\varphi$ are infinitely near $p_1$ and there do not exist $p_i, p_j$ among them such that $p_j\infnear_1 p_i \infnear_1 p_1$ and $p_j$ is not proximate to $p_1$.
Then, the weighted proximity graph of $\varphi$ is one of the following three graphs:
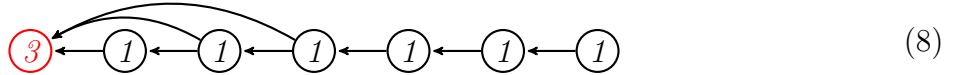
\begin{align}
&\scalebox{0.5}{%
\begin{tikzpicture}[->,>=stealth',shorten >=2pt,auto,node distance=2.5cm,ultra thick,baseline=0]
  \tikzstyle{every state}=[fill=white,draw=black,text=black]
  \node[state,draw=red,text=red] (A)    {\huge 3};
  \node[state] (B) [right of=A] {\huge 1};
  \node[state] (C) [right of=B] {\huge 1};
  \node[state] (D) [right of=C] {\huge 1};
  \node[state] (E) [right of=D] {\huge 1};
  \node[state] (F) [right of=E] {\huge 1};
  \node[state] (G) [right of=F] {\huge 1};
\path (B) edge (A);
\path (C) edge (B);
\path (D) edge (C);
\path (E) edge (D);
\path (F) edge (E);
\path (G) edge (F);
\path (C) edge [bend right] (A);
\path (D) edge [bend right] (A);
\addvmargin{1mm}
\end{tikzpicture}%
}
\label{graph1.1} \\
&\scalebox{0.5}
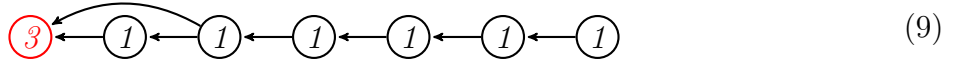
{%
\begin{tikzpicture}[->,>=stealth',shorten >=2pt,auto,node distance=2.5cm,ultra thick,baseline=0]
  \tikzstyle{every state}=[fill=white,draw=black,text=black]
  \node[state,draw=red,text=red] (A)    {\huge 3};
  \node[state] (B) [right of=A] {\huge 1};
  \node[state] (C) [right of=B] {\huge 1};
  \node[state] (D) [right of=C] {\huge 1};
  \node[state] (E) [right of=D] {\huge 1};
  \node[state] (F) [right of=E] {\huge 1};
  \node[state] (G) [right of=F] {\huge 1};
\path (B) edge (A);
\path (C) edge (B);
\path (D) edge (C);
\path (E) edge (D);
\path (F) edge (E);
\path (G) edge (F);
\path (C) edge [bend right] (A);
\addvmargin{1mm}
\end{tikzpicture}%
}
\label{graph2.1} \\
&\scalebox{0.5}
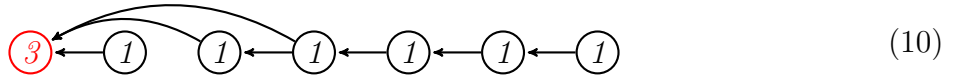
{%
\begin{tikzpicture}[->,>=stealth',shorten >=2pt,auto,node distance=2.5cm,ultra thick,baseline=0]
  \tikzstyle{every state}=[fill=white,draw=black,text=black]
  \node[state,draw=red,text=red] (A)    {\huge 3};
  \node[state] (B) [right of=A] {\huge 1};
  \node[state] (C) [right of=B] {\huge 1};
  \node[state] (D) [right of=C] {\huge 1};
  \node[state] (E) [right of=D] {\huge 1};
  \node[state] (F) [right of=E] {\huge 1};
  \node[state] (G) [right of=F] {\huge 1};
\path (B) edge (A);
\path (D) edge (C);
\path (E) edge (D);
\path (F) edge (E);
\path (G) edge (F);
\path (C) edge [bend right] (A);
\path (D) edge [bend right] (A);
\addvmargin{1mm}
\end{tikzpicture}%
}
\label{graph4.1}
\end{align}
\end{lemma}

\begin{proof}
Recall that the point $p_1$ may have at most three base points of multiplicity 1 that are proximate to $p_1$.
Moreover, a base point of multiplicity 1 cannot have more than one point that is proximate to it.

Suppose that the base points $p_2, \ldots, p_7$ of multiplicity 1 are such that: $p_7 \infnear_1 p_6 \infnear_1 p_5 \infnear_1 p_4 \infnear_1 p_3 \infnear_1 p_2 \infnear_1 p_1$. The assumptions imply that $p_3$ must be proximate to $p_1$.
In case $p_4$ is proximate to $p_1$ too, the weighted proximity graph is isomorphic to \eqref{graph1.1}.
Otherwise, if $p_4$ is not proximate to $p_1$, the weighted proximity graph is isomorphic to \eqref{graph2.1}.

Otherwise, there are at least two base points, say $p_2$ and $p_3$, of multiplicity 1 that are infinitely near of order 1 to $p_1$. Then, either $p_4$ is infinitely near of order 1 to $p_1$ or $p_4$ is infinitely near one among $p_2$ and $p_3$, say $p_3$.

In the former case, $p_5$ must be infinitely near one among $p_2, p_3, p_4$; since $p_5$ cannot be proximate to $p_1$, we get a contradiction to the hypothesis.

In the latter case, $p_4$ must be proximate to $p_1$. The base point $p_5$ cannot be proximate to $p_2$, because in that case $p_5$ would not be proximate to $p_1$, a contradiction to the hypothesis. Therefore, $p_5 \infnear_1 p_4$ and similarly $p_7 \infnear_1 p_6 \infnear_1 p_5$.
We conclude that the weighted proximity graph is isomorphic to \eqref{graph4.1}.
\end{proof}

In the following propositions, we classify plane Cremona maps of degree 4 with the weighted proximity graphs found in the previous lemma.

\begin{proposition}
Let $\varphi$ be a plane Cremona map with weighted proximity graph \eqref{graph1.1}. Then $\varphi$ is equivalent to the following two maps 
\begin{equation}\label{phi1.1}
\varphi_1([x:y:z]) = [xz^3+y^4+ky^2z^2 : yz^3 : z^4],
\end{equation}
where $k$ is either $0$ or $1$.
\end{proposition}

\begin{proof}
Notice that $\varphi_1$ has the following base points:
\begin{itemize}
    \item $p_1$ of multiplicity 3 at $[1:0:0]$;
    \item $p_2\infnear_1 p_1$ with standard coordinates $(p_1,0)$, namely in the direction of the line $z=0$;
    \item $p_3 \infnear_1 p_2$ is proximate to $p_1$, so that it has standard coordinates $(p_1,0,\infty)$;
    \item $p_4\infnear_1 p_3$ is proximate to $p_1$ too, so that it has standard coordinates $(p_1,0,\infty,0)$;
    \item $p_5\infnear_1 p_4$ has standard coordinates $(p_1,0,\infty,0,-1)$;
    \item $p_6\infnear_1 p_5$ has standard coordinates $(p_1,0,\infty,0,-1,0)$;
    \item $p_7\infnear_1 p_6$ has standard coordinates $(p_1,0,\infty,0,-1,0,k)$.
\end{itemize}
It suffices to show that, if $\varphi$ has weighted proximity graph \eqref{graph1.1}, then there exists an automorphism of $\PP^2$ mapping the base points $q_1,\ldots,q_7$ of $\varphi$ to the points with the above standard coordinates.
Indeed, automorphisms $\alpha$ of $\PP^2$ that fix $p_1,p_2,p_3,p_4$ and send $q_5$ with standard coordinates $(p_1,0,\infty,0,s)$ ($s\in\K$ and $s\ne0$) to $p_5$ with standard coordinates $(p_1,0,\infty,0,-1)$ are such that
\[
\alpha^{-1}([x:y:z]) = [ay + cz + x: by + dz: (-sb)^{1/3}bz]
\]
for some $a,b,c,d\in\K$ and $b\ne0$.
Automorphisms $\beta$ of $\PP^2$ that fix $p_1,p_2,p_3,p_4,p_5$ and send $q_6$ with standard coordinates $(p_1,0,\infty,0,-1,r)$, where $r\in\K$, to $p_6$ with standard coordinates $(p_1,0,\infty,0,-1,0)$ are such that
\[
\beta^{-1}([x:y:z]) = [ay + cz + x: by + 1/4brz: b^{4/3}z]
\]
for some $a,b,c\in\K$ and $b\ne0$.
Automorphisms $\gamma$ of $\PP^2$ that fix $p_1,p_2,p_3,p_4,p_5,p_6$ and send $q_7$ with standard coordinates $(p_1,0,\infty,0,-1,0,t)$, where $t\in\K$ and $t\neq 0$, to $p_7$ with standard coordinates $(p_1, 0, \infty, 0, -1, 0, 1)$ are such that
\[
\gamma^{-1}([x:y:z]) =[ay + cz + x: t^{3/2}y: {(t^{3/2})}^{4/3}z]
\]
for some $a,c\in\K$.
Moreover, one can check that there is no automorphism that fixes $p_1,p_2,p_3,p_4,p_5,p_6$ and sends the point which has standard coordinates $(p_1,0,\infty,0,-1,0,0)$ to the point with standard coordinates $(p_1,0,\infty,0,-1,0,1)$.
\end{proof}

\begin{proposition}
Let $\varphi$ be a plane Cremona map with weighted proximity graph \eqref{graph2.1}. Then $\varphi$ is equivalent to the map 
\begin{equation}\label{phi2.1}
\varphi_2([x:y:z]) = [y(xz^2 + y^3) : z(xz^2 + y^3) : z^4]. 
\end{equation}
\end{proposition}

\begin{proof}
Notice that $\varphi_2$ has the following base points:
\begin{itemize}
    \item $p_1$ of multiplicity 3 at $[1:0:0]$;
    \item $p_2\infnear_1 p_1$ with standard coordinates $(p_1,0)$, namely in the direction of the line $z=0$;
    \item $p_3 \infnear_1 p_2$ is proximate to $p_1$, so that it has standard coordinates $(p_1,0,\infty)$;
    \item $p_4\infnear_1 p_3$ has standard coordinates $(p_1,0,\infty,-1)$;
    \item $p_5\infnear_1 p_4$ has standard coordinates $(p_1,0,\infty,-1,0)$;
    \item $p_6\infnear_1 p_5$ has standard coordinates $(p_1,0,\infty,-1,0,0)$;
    \item $p_7\infnear_1 p_6$ has standard coordinates $(p_1,0,\infty,-1,0,0,0)$.
\end{itemize}
It suffices to show that, if $\varphi$ has weighted proximity graph \eqref{graph2.1}, then there exists an automorphism of $\PP^2$ mapping the base points $q_1,\ldots,q_7$ of $\varphi$ to the points with the above standard coordinates.
In \cite[Lemma 9.1]{CN} we proved that there exists an automorphism of $\PP^2$ mapping $q_1,\ldots,q_6$ to $p_1,\ldots,p_6$ as above.
Furthermore, one may compute that automorphisms $\alpha$ of $\PP^2$ that fix $p_1,\ldots,p_6$ are such that
\[
\alpha^{-1}([x:y:z]) = [x + az : by : \pm b^{3/2} z]
\]
for some $a,b\in\K$, $b\ne0$.
A cubic curve passing through $p_1, \ldots, p_6$ and $q_7$ with standard coordinates $(p_1,0,\infty,-1,0,0,s)$, where $s\in\K$, is sent to the cubic curve $xz^2+y^3$ by the automorphism $\beta$ of $\PP^2$ such that
\[
\beta^{-1}([x:y:z]) = [x \mp s b^{3/2}z : by : \pm b^{3/2} z]
\]
for some $b\in\K$ and $b\neq 0$.
We conclude that, up to automorphism of $\PP^2$, we may assume that $p_7$ has standard coordinates $(p_1,0,\infty,-1,0,0,0)$.
\end{proof}

\begin{proposition}
Let $\varphi$ be a plane Cremona map with weighted proximity graph \eqref{graph4.1}. Then $\varphi$ is equivalent to the following two maps
\begin{equation}\label{phi4.1}
\varphi_3([x:y:z]) = [(y + kz)(xz^2 + y^3) : yz^3 : z^4],
\end{equation}
where $k$ is either $0$ or $1$.
\end{proposition}

\begin{proof}
Notice that $\varphi_3$ has the following base points:
\begin{itemize}
    \item $p_1$ of multiplicity 3 at $[1:0:0]$;
    \item $p_2\infnear_1 p_1$ with standard coordinates either $(p_1,-1)$ if $k=1$ (i.e.\ in the direction of the line $y+z=0$) or $(p_1,\infty)$ if $k=0$ (i.e.\ in the direction of the line $y=0$);
    \item $p_3 \infnear_1 p_1$  with standard coordinates $(p_1,0)$, namely in the direction of the line $z=0$;
    \item $p_4\infnear_1 p_3$ is proximate to $p_1$, so it has standard coordinates $(p_1,0,\infty)$;
    \item $p_5\infnear_1 p_4$ has standard coordinates $(p_1,0,\infty,-1)$;
    \item $p_6\infnear_1 p_5$ has standard coordinates $(p_1,0,\infty,-1,0)$;
    \item $p_7\infnear_1 p_6$ has standard coordinates $(p_1,0,\infty,-1,0,0)$.
\end{itemize}
It suffices to show that, if $\varphi$ has weighted proximity graph \eqref{graph4.1}, then there exists an automorphism of $\PP^2$ mapping the base points $q_1,\ldots,q_7$ of $\varphi$ to the points with the above standard coordinates.
As in the proof of the previous proposition, we know that automorphisms $\alpha$ of $\PP^2$ that fix $p_1, p_3, p_4, p_5, p_6, p_7$ are such that
\[
\alpha^{-1}([x:y:z]) = [x + az : by : \pm b^{3/2} z]
\]
for some $a,b\in\K$, $b\ne0$.
If $q_2 \infnear_1 p_1$ has standard coordinates $(p_1,s)$, with $s\in\K^*$, namely $q_2$ is in the direction of the line $sy-z=0$, then the line $sy-z=0$ is sent to the line $y+z=0$ by the automorphism $\beta$ of $\PP^2$ such that
\[
\beta^{-1}([x:y:z]) = [x+az : s^2 y : \pm s^3 z]
\]
for some $a\in\K$.
We conclude that, up to automorphism of $\PP^2$, we may assume that $p_2$ has standard coordinates $(p_1,-1)$.
Finally, one can check that there is no automorphism that fixes $p_1,p_3,p_4,p_5,p_6,p_7$ and sends the line $y=0$ to the line $y+z=0$; in other words, there is no automorphism that fixes $p_1,p_3,p_4,p_5,p_6,p_7$ and sends the point which has standard coordinates $(p_1,\infty)$ to the point with standard coordinates $(p_1,-1)$.
\end{proof}

It is easy to find explicit decompositions of the previous maps.

\begin{example}\label{phi1.1_decom}
A decomposition of \eqref{phi1.1} into $5$ quadratic maps is
\begin{align*}
\varphi_1 &=[8k(z-y) + x - 8y + 6z: -8y + 4z: -8y] \circ \textcolor{red}{\rho}  \circ[2k(z-2y) - x - 2y + z: 2y - z: y]  \,\circ\,\\
&\phantom{ = } \circ \textcolor{red}{\rho} \circ[-x - 2z: -2x + 2y: x]  \circ \textcolor{red}{\tau}  \circ [z: y + 2z: -x + z + y]  \circ \textcolor{red}{\rho}  \,\circ\,\\
&\phantom{ = } \circ [x + 2z - y: z: y - z] \circ \textcolor{red}{\rho}  \circ [x - y: z: y]
\end{align*}
where $k$ is either $0$ or $1$.
\end{example}

\begin{example}\label{phi2.1_decom}
A decomposition of \eqref{phi2.1} into $4$ quadratic maps is
\begin{align*}
\varphi_2 &= [x : z : y] \circ \textcolor{red}{\rho}  \circ
[z : x : y]  \circ \textcolor{red}{\rho}  \circ
[z : y : x]  \circ \textcolor{red}{\tau}  \circ
[z : y : -x + z]  \circ \textcolor{red}{\rho}  \circ
[x + y : z: y].
\end{align*}
\end{example}

\begin{example}
A decomposition of \eqref{phi4.1} into $4$ quadratic maps is
\begin{align*}
\varphi_3 &= [-kz + x: -ky - z: y] \circ \textcolor{red}{\rho}  \circ
[ky + x - y + z: -ky - y + z: y]  \circ \textcolor{red}{\rho}  \,\circ\, \\
&\phantom{ = } \circ [y + z: x - y: x]  \circ \textcolor{red}{\tau} \circ [-z: y - z: -x]  \circ \textcolor{red}{\rho}  \circ[-x: -z: y - z]
\end{align*}
where $k$ is either $0$ or $1$.
\end{example}

Next, we will find the quadratic lengths of all such maps.

\begin{lemma}
The quadratic length of $\varphi_2$ is $\q(\varphi_2) =4$. 
\end{lemma}

\begin{proof}
Since $\q(\varphi_2)=\q(\varphi^{-1}_2)$ according to Remark \ref{rem:inverse}, it is sufficient to prove that $\q(\varphi^{-1}_2) = 4$. Indeed, the inverse map $\varphi^{-1}_2$ is the following map
\begin{equation*}\label{phi2.1_inver}
\varphi_2^{-1}([x:y:z]) = [-x^{3} z +y^{4}: x \,y^{2} z: z \,y^{3}]
\end{equation*}
and it has enriched weighted proximity graph
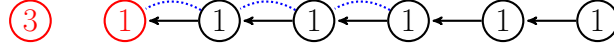
\[
\scalebox{0.5}{%
\begin{tikzpicture}[->,>=stealth',shorten >=2pt,auto,node distance=2.5cm,ultra thick]
  \tikzstyle{every state}=[fill=white,draw=black,text=black]
  \node[state,draw=red,text=red] (1)  {\huge 3};
  \node[state,draw=red,text=red] (2) [right of=1] {\huge 1};
  \node[state] (3) [right of=2] {\huge 1};
  \node[state] (4) [right of=3] {\huge 1};
  \node[state] (5) [right of=4] {\huge 1};
  \node[state] (6) [right of=5] {\huge 1};
  \node[state] (7) [right of=6] {\huge 1};
\path (7) edge (6);
\path (6) edge (5);
\path (5) edge (4);
\path (4) edge (3);
\path (3) edge (2);
\path[-,dotted,blue] (5) edge [bend right] (4);
\path[-,dotted,blue] (4) edge [bend right] (3);
\path[-,dotted,blue] (3) edge [bend right] (2);
\addvmargin{1mm}
\end{tikzpicture}%
}
\]
where the blue dotted line means that the corresponding base points $p_2, p_3, p_4, p_5$ are collinear. By Lemma \ref{lem5.15} and Lemma \ref{lem:5.4}, we have $$3\leq \q(\varphi^{-1}_2) \leq 4.$$

Suppose by contradiction that $\q(\varphi _{2}^{-1}) = 3$. Then, there should exist a quadratic map $\rho_1$ such that $\q(\varphi _{2}^{-1} \circ \rho_1^{-1}) = 2$. Note that, $\rho_1$ must be based at $p_1,p_2$ and $p_3$. Otherwise, by Proposition \ref{PropMAC118b}, one has $\deg(\varphi _{2}^{-1} \circ \rho_1^{-1}) \geq 4$. If $\deg(\varphi _{2}^{-1} \circ \rho_1^{-1}) \geq 5$, a contradiction with Lemma \ref{cor5.14}. On the other hand, $\deg(\varphi _{2}^{-1} \circ \rho_1^{-1}) =4$ if and only if $\rho_1$ is based at $p_1,p_2$ and a general point, in this case the composition map $\varphi _{2}^{-1} \circ \rho_1^{-1}$ is a de \Jonq\ map, a contradiction with Lemma \ref{lem5.15}. However, with $\rho_1$ having base points at $p_1,p_2$ and $p_3$, the composition map $\varphi _{2}^{-1} \circ \rho_1^{-1}$ is a plane Cremona map of degree 3 of type 1 in our classification in \cite{CN}, which has been shown in Theorem \ref{thm3.1} that its quadratic length is 3, a contradiction.
\end{proof}
As a consequence, we can deduce that
\begin{corollary}
 Let $\varphi$ be a plane Cremona map with weighted proximity graph \eqref{graph2.1}. Then, one has $\q(\varphi)=4$.
\end{corollary}

The proof of the next lemma is exactly the same. Accordingly, we also prove that the corresponding inverse map has quadratic length of 4.
\begin{lemma}
The quadratic length of $\varphi_3$ is $\q(\varphi_3) =4$. 
\end{lemma}

\begin{proof}
The inverse map $\varphi^{-1}_3$ of the map $\varphi_3$ is the following map
\begin{equation*}\label{phi4.1_inver}
\varphi_3^{-1}([x:y:z]) = [-k \,y^{3} z +x \,z^{3}-y^{4}: y \,z^{2} \left(k z +y \right):  \left(k z +y \right) z^{3}]
\end{equation*}
where $k$ is either $0$ or $1$, and it has weighted proximity graph
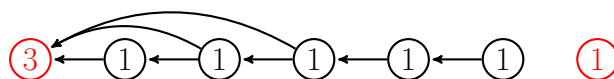
\[
\scalebox{0.5}{%
\begin{tikzpicture}[->,>=stealth',shorten >=2pt,auto,node distance=2.5cm,ultra thick]
  \tikzstyle{every state}=[fill=white,draw=black,text=black]
  \node[state,draw=red,text=red] (1)  {\huge 3};
  \node[state] (2) [right of=1] {\huge 1};
  \node[state] (3) [right of=2] {\huge 1};
  \node[state] (4) [right of=3] {\huge 1};
  \node[state] (5) [right of=4] {\huge 1};
  \node[state] (6) [right of=5] {\huge 1};
  \node[state,draw=red,text=red] (7) [right of=6] {\huge 1};
\path (6) edge (5);
\path (5) edge (4);
\path (4) edge (3);
\path (3) edge (2);
\path (2) edge (1);
\path (4) edge [bend right] (1);
\path (3) edge [bend right] (1);
\addvmargin{1mm}
\end{tikzpicture}%
}
\]
where the base points from left to right respectively are $p_1,p_2,p_3,p_4,p_5,p_6$ and $p_7$.  By Lemma \ref{lem5.15} and Lemma \ref{lem:5.4}, we have $$3\leq \q(\varphi^{-1}_3) \leq 4.$$

Suppose by contradiction that $\q(\varphi _{3}^{-1}) = 3$. Then, there should exist a quadratic map $\rho_1$ such that $\q(\varphi _{3}^{-1} \circ \rho_1^{-1}) = 2$. Note that $\rho_1$ must be based at $p_1,p_2$ and $p_7$. Otherwise, by Proposition \ref{PropMAC118b}, one has $\deg(\varphi _{3}^{-1} \circ \rho_1^{-1}) \geq 4$. If $\deg(\varphi _{3}^{-1} \circ \rho_1^{-1}) \geq 5$, a contradiction with Lemma \ref{cor5.14}. On the other hand, $\deg(\varphi _{3}^{-1} \circ \rho_1^{-1}) =4$ if and only if $\rho_1$ is based at $p_1,p_i$ (where $i$ is either $2$ or $7$) and a general point, in this case the composition map $\varphi _{3}^{-1} \circ \rho_1^{-1}$ is a de \Jonq\ map, a contradiction with Lemma \ref{lem5.15}. However, with $\rho_1$ having base points at $p_1,p_2$ and $p_7$, the composition map $\varphi _{3}^{-1} \circ \rho_1^{-1}$ is a cubic map of type 1 in our classification in \cite{CN}, which has been shown in Theorem \ref{thm3.1} that its quadratic length is 3, a contradiction.
\end{proof}

\begin{corollary}\label{corT4.1}
Let $\varphi$ be a plane Cremona map with weighted proximity graph \eqref{graph4.1}. Then, one has $\q(\varphi)=4$.
\end{corollary}

\begin{lemma}
The quadratic length of $\varphi_{1}$ is $\q(\varphi_1)=5$.
\end{lemma}
\begin{proof}
Example \ref{phi1.1_decom} implies that $\q(\varphi_{1}) \leq 5$ and Lemma \ref{lem5.15} that
$$3 \leq \q(\varphi_{1}) \leq 5.$$
To prove that $\q(\varphi_{1})=5$, it is sufficient to show that $\q(\varphi_{1})$ cannot equal 3 or 4.
\begin{itemize}
\item First of all, we prove that $\q(\varphi_{1})$ cannot be 3.\\
Indeed, suppose by contradiction that $\q(\varphi_{1}) = 3$. Then, there should exist a quadratic map $\rho_1$ such that $\q(\varphi_{1} \circ \rho_1^{-1}) = 2$. By Proposition \ref{PropMAC118b}, one has $\deg(\varphi_{1} \circ \rho_1^{-1}) \geq 4$. If $\deg(\varphi_{1} \circ \rho_1^{-1}) \geq 5$, a contradiction with Lemma \ref{cor5.14}. Otherwise, $\deg(\varphi_{1} \circ \rho_1^{-1}) =4$ if and only if $\rho_1$ is based at $p_1,p_2$ and a general point, however in this case the composition map $\varphi _1 \circ \rho_1^{-1}$ is a de \Jonq\ map which has weighted proximity graph \eqref{graph4.1}, a contradiction with Corollary \ref{corT4.1}.
\item Now, suppose that $\q(\varphi_{1}) = 4$. Then, there should exist a quadratic map $\rho_1$ such that $\q(\varphi_{1} \circ \rho_1^{-1}) = 3$. In particular, $\rho_1$ must be based at $p_1$, otherwise by Proposition \ref{PropMAC118b}, one has $\deg(\varphi_{1} \circ \rho_1^{-1}) = 8$ and its proximity graph with the corresponding base points is of the following form:
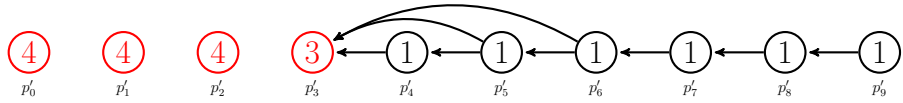
\begin{center}
\scalebox{0.5}{%
\begin{tikzpicture}[->,>=stealth',shorten >=2pt,auto,node distance=2.5cm,ultra thick,baseline=0]
  \tikzstyle{every state}=[fill=white,draw=black,text=black]

\node[state,draw=red,text=red, label=below: $p'_0$] (A)    {\huge 4};
  \node[state,draw=red,text=red] (B) [right of=A,label=below: $p'_1$] {\huge 4};
  \node[state,draw=red,text=red] (C) [right of=B,label=below: $p'_2$] {\huge 4};
 
  \node[state,draw=red,text=red] (D) [right of=C,label=below: $p'_3$]  {\huge 3};
  \node[state] (E) [right of=D,label=below: $p'_4$] {\huge 1};
  \node[state] (F) [right of=E,label=below: $p'_5$] {\huge 1};
  \node[state] (G) [right of=F,label=below: $p'_6$] {\huge 1};
  \node[state] (H) [right of=G,label=below: $p'_7$] {\huge 1};
  \node[state] (I) [right of=H,label=below: $p'_8$] {\huge 1};
  \node[state] (K) [right of=I,label=below: $p'_9$] {\huge 1};

\path (E) edge (D);
\path (F) edge (E);
\path (G) edge (F);
\path (H) edge (G);
\path (I) edge (H);
\path (K) edge (I);
\path (F) edge [bend right] (D);
\path (G) edge [bend right] (D);
\addvmargin{1mm}
\end{tikzpicture}%
}
\\
\end{center}
Then, there should exist a quadratic map $\rho_2$ such that $\q(\varphi_{1} \circ \rho_1^{-1}\circ \rho_2^{-1}) = 2$. However, by Proposition \ref{PropMAC118b}, one has $\deg(\varphi_{1} \circ \rho_1^{-1}\circ \rho_2^{-1}) \geq 4$. If $\deg(\varphi_{1} \circ \rho_1^{-1}\circ \rho_2^{-1}) \geq 5$, a contradiction with Lemma \ref{cor5.14}. Otherwise, $\deg(\varphi_{1} \circ \rho_1^{-1}\circ \rho_2^{-1}) =4$ if and only if $\rho_2$ is based at $p'_0, p'_1$ and $p'_2$, moreover in this case the composition map $\varphi_{1} \circ \rho_1^{-1}\circ \rho_2^{-1}$ is a de \Jonq\ map which has weighted proximity graph \eqref{graph1.1}, a contradiction with the hypothesis.

If $\rho_1$ is based at $p_1$ and is not based at $p_2$, then $\deg(\varphi_{1} \circ \rho_1^{-1}) = 5$ and it is a de \Jonq\ map,  a contradiction with  Lemma \ref{lem5.15}.

Thus, $\rho_1$ must be based at $p_1$ and $p_2$.  So the map $\varphi_{1} \circ \rho_1^{-1}$ is a quartic de \Jonq\ map which has weighted proximity graph \eqref{graph4.1}, a contradiction with Corollary \ref{corT4.1}.
\end{itemize}
We conclude that $\q(\varphi_{1})=5$ as claimed.
\end{proof}

\begin{corollary}
Let $\varphi$ be a plane Cremona map with weighted proximity graph \eqref{graph1.1}. Then, one has $\q(\varphi)=5$.
\end{corollary}

\subsection*{Acknowledgements}
This work was completed while the second author was a postdoctoral researcher at Vietnam Institute for Advanced Studies in Mathematics (VIASM) and Tokyo University of Science (TUS). She sincerely thanks VIASM and TUS for their very kind support and hospitality.
The first author is supported by the Departmental Research Funding [FIRD] from the Department of Mathematics and Computer Science of the University of Ferrara.
The second author is supported by Project for Mathematical Science and Cooperation with Engineering (MaSCE) at TUS.
We would like to thank the anonymous referee for his/her careful reading, relevant corrections, and helpful suggestions.

\end{document}